\documentclass[article]{amsart}
\usepackage{amsmath, amssymb, amsthm, bm}
\usepackage{color}
\usepackage{tikz}
\usepackage{tikz-cd}
\usepackage{mathrsfs}
\usepackage{graphicx,subfig,caption}
\usepackage{mathtools}

\addtocounter{section}{-1}

\def\C{\mathbb C}
\def\R{\mathbb R}

\def\Z{\mathbb Z}

\def\ge{\geqslant}

\theoremstyle{plain}
\newtheorem{thm}{Theorem}[section]
\newtheorem{lemma}[thm]{Lemma}
\newtheorem{cor}[thm]{Corollary}

\newtheorem{cl}[thm]{Claim}

\theoremstyle{definition}
\newtheorem{defn}[thm]{Definition}

\begin{document}

\title[Combinatorial proof of the Maslov index formula]{Combinatorial proof of Maslov index formula in Heegaard Floer Theory}

\author{Roman Krutowski}
\address{Department of Mathematics, University of California, Los Angeles}
\email{romankrut@ucla.edu}

\begin{abstract}
We prove Lipshitz's Maslov index formula in Heegaard Floer homology via the combinatorics of Heegaard diagrams.
\end{abstract}

\maketitle

\setcounter{section}{0}

\section{Introduction}

In \cite{OzSz2004b, OzSz2004c} Peter Ozsv\'{a}th and Zolt\'{a}n Szab\'{o} introduced Heegaard Floer homology, a collection of invariants for closed oriented 3-manifolds. Since then Heegaard Floer homology has emerged as an extremely powerful invariant, producing strong results in low-dimensional topology (see \cite{OzSz2004b, OzSz2004c, OzSz2006, OzSz2004g, Ni2009, Ghi2008, Ni2009f}).  
All of the numerous versions of Heegaard Floer homology involve counting the number of points in the
unparametrized moduli space of $J$-holomorphic disks of a certain Maslov index joining some
intersection points of two half-dimensional tori in a certain symmetric product of a surface. 

The most notable advantage of Heegaard Floer homology when compared to other types of Floer Homology is its combinatorial nature which allows computation of these invariants in various cases. One of the ingredients is the combinatorial formula for Maslov index $\mu(\varphi)$ of a Whitney disk $\varphi$, which is the Fredholm index of some differential operator, or alternatively, a homology class of a certain loop in the Lagrangian Grassmanian. In \cite{Ras2003}, Jacob Rasmussen gave a formula that relates the intersection
number of the disk with the fat diagonal in the symmetric product, with its Maslov index. Later on, Robert Lipshitz in the paper devoted to the cylindrical reformulation of the whole theory \cite{Lip2006} gave a purely combinatorial formula for the Maslov index of these disks.

The proof of this formula in \cite{Lip2006, Lip2014} is based on an elegant geometric approach. In this paper, we provide a combinatorial proof of this formula which is inspired by the proof of the index formula for Maslov $n$-gons due to Sucharit Sarkar \cite{Sar2011}. 

\smallskip{}
Let $(\Sigma, \bm \alpha, \bm \beta)$ be a Heegaard diagram. Any Whitney disk $\varphi \in \pi_2(\bm x, \bm y)$ connecting $\bm x, \bm y \in Sym^g(\Sigma)$ has a \emph{shadow} $D(\varphi)$ (see Definition~\ref{shadow}), a certain $2$-chain in $\Sigma$ with boundary  satisfying $\partial(\partial D(\varphi) \cap \bm \alpha)= \bm y - \bm x$  and  $\partial(\partial D(\varphi) \cap \bm \beta)= \bm x - \bm y$. We denote the set of such $2$-chains by $\mathcal{D}(\bm x, \bm y)$ and call them \emph{domains}. We denote by $\mathcal{D}$ the set of all domains in all Heegaard diagrams (see Section~\ref{not} for details).

In this paper, we assume that $\bm \alpha$ curves intersect $\bm \beta$ curves perpendicularly with respect to some metric on $\Sigma$. Among the domains, there are two special types that serve us as building blocks. \emph{Bigons} and \emph{rectangles} are $2$-sided and $4$-sided domains, respectively, which are homeomorphic to disks with all angles equal to $90^\circ$ and which do not contain any $\bm x$- and $\bm y$-coordinates in the interior.

Given two domains $D \in \mathcal{D}(\bm x, \bm y)$ and $D' \in \mathcal{D}(\bm y, \bm z)$ the composition of these domains $D*D'$ is a domain $D+D' \in \mathcal{D}(\bm x, \bm z)$.
 
In this paper any map $\overline{\mu} \colon \mathcal{D} \rightarrow \Z $ will be called an \emph{index}.
An index $\overline{\mu}$ is said to be \emph{additive} if for any  Heegaard diagram and any two of its domains $D \in \mathcal{D}(\bm x, \bm y)$ and $D' \in \mathcal{D}(\bm y, \bm z)$ the following holds
\[
\overline{\mu}(D*D')=\overline{\mu}(D)+\overline{\mu}(D').
\]
In Section~\ref{not} we introduce two types of transformations of any Heegaard diagram $(\Sigma, \bm \alpha, \bm \beta)$ which assign to each of its domain $D$ a new domain $D'$ in the new Heegaard diagram. These transformations are called \emph{finger moves} and \emph{empty stabilizations}.
An index $\overline{\mu}$ is said to be \emph{stable} if for any such transformation $\overline{\mu}(D)=\overline{\mu}(D')$.

Define the \emph{combinatorial index} $\widetilde{\mu}$ of a domain $D \in \mathcal{D}(\bm x, \bm y)$ via the formula due to Lipshitz \cite{Lip2006}
\begin{equation}\label{combind}
\widetilde{\mu}(D) \coloneqq \widetilde{\mu}_{\bm x}(D)+\widetilde{\mu}_{\bm y}(D)+e(D).
\end{equation}
where $\widetilde{\mu}_{\bm x}(D)$ is a point measure of $D$ at $\bm x$ and $e(D)$ is the Euler measure of $D$ (see Subsection~\ref{prelim}).

We are now ready to formulate our main results.

\begin{thm}\label{unique}
There exists a unique index $\overline{\mu} \colon \mathcal{D} \rightarrow \Z$ satisfying the following axioms:

\begin{enumerate}
    \item $\overline{\mu}$ is additive;
    \item $\overline{\mu}$ is stable;
    \item $\overline{\mu}(B)=1$ for any bigon $B \in \mathcal{D}$;
    \item $\overline{\mu}(R)=1$ for any rectangle $R \in \mathcal{D}$.
\end{enumerate}
Moreover, this index agree with the combinatorial index $\widetilde{\mu}$. 
\end{thm}

\begin{thm}\label{Lipshitz}
Let $(\Sigma, \bm \alpha, \bm \beta)$ be a Heegaard diagram and $\varphi$ be a Whitney disk connecting two generators of the corresponding Heegaard Floer chain complex. Then  
\[
\mu(\varphi)=\widetilde{\mu}(D(\varphi)),
\]
where $\mu(\varphi)$ is the Maslov index of $\varphi$. 
\end{thm}

\medskip
{\bf Acknowledgments.} I would like to thank Sucharit Sarkar for his guidance and encouragement, Ko Honda for helpful conversations and suggestions, and Gleb Terentiuk for valuable comments.

\section{Notations and preliminary results}\label{not}

\subsection{Heegaard Floer Theory preliminaries}\label{prelim}

To start, recall the definition of a Heegaard diagram. Let $\Sigma$ be an oriented Riemannian surface (i.e., there is a fixed Riemannian metric) and $\bm \alpha= \{\alpha_1, \ldots, \alpha_g\}$ and $\bm \beta= \{\beta_1, \ldots, \beta_g\}$ be two sets of non-intersecting \emph{oriented}\footnote{This is nonstandard but useful to assume for the entirety of the paper.} simple closed curves such that both $\bm \alpha$ and $\bm \beta$ generate half-dimensional subspaces of $H_1(\Sigma)$ (in particular this means that $g$ is not smaller than the genus $g(\Sigma)$ of the surface). The last part is equivalent to assuming that the complement of $\bm \alpha$ consists of $g-g(\Sigma)+1$ components (the same holds for $\bm \beta)$.  We also assume that the $\bm \alpha$ and $\bm \beta$ curves intersect perpendicularly. We call by \emph{regions} closures of connected components of $\Sigma \setminus (\bm \alpha \cup \bm \beta)$. The collection $(\Sigma, \bm \alpha, \bm \beta)$ is usually called an \emph{unpointed Heegaard diagram}, but in the text we refer to it simply as a \emph{Heegaard diagram}. We need not assume these diagrams to be pointed, because it is not relevant for the Maslov index calculations.

The Heegaard Floer homology chain complex of a diagram $(\Sigma, \bm \alpha, \bm \beta)$ is generated by $g$-tuples of points of the form $\bm x= \{x_1, \ldots, x_g\}$ where $x_i \in \alpha_i \cap \beta_{\sigma(i)}$ and $\sigma \in S_g$ is arbitrary. We may regard $\bm x$ as a point in $Sym^g(\Sigma)$ which belongs to $T_{\bm \alpha} \cap T_{\bm \beta}$ where $T_{\bm \alpha}=\alpha_1\times \alpha_2 \times \ldots \times \alpha_g$ and $T_{\bm \beta}=\beta_1\times \beta_2 \times \ldots \times \beta_g$. 

Let as consider the unit disk $D^2$ in $\C$ with the usual orientation. Let $s_1 \subset \partial D^2$ be the portion of the oriented boundary that connects $i$ to $-i$ and let $s_2\subset \partial D^2$ be the remaining portion of the boundary connecting $-i$ to $i$. The differentials and chain maps are signed counts of $J$-holomorphic maps $u: D^2 \setminus \{i, -i\} \rightarrow Sym^g(\Sigma)$ representing Whitney disks in $\pi_2(\bm x, \bm y)$. Here a Whitney disk is a homotopy type of maps from $D^2$ to $Sym^g(\Sigma)$ such that $s_1$ is mapped to $T_{\bm \alpha}$, $s_2$ is mapped to $T_{\bm \beta}$ and $i$, $-i$ are sent to $\bm x$ and $\bm y$, respectively.

Let $D$ be a $2$-chain obtained as a sum of regions of a Heegaard diagram $(\Sigma, \bm \alpha, \bm \beta)$ with integer coefficients, $D=\sum_{R}a(R)R$ where $a(R) \in \Z$. Such a $2$-chain $D$ is called a \emph{domain} if there exist two generators $\bm x$ and $\bm y$ such that $ \partial (\partial D|_{\bm \alpha}) \coloneqq \partial( \partial D \cap \bm \alpha)= \bm y - \bm x$ (here we regard $\bm x$ and $\bm y$ as $0$-chains in $\Sigma$) and  $\partial(\partial D|_{\bm \beta}) = \bm x - \bm y$. We say that $D$ \emph{connects} $\bm x$ and $\bm y$ and denote the set of these domains by $\mathcal{D}(\bm x, \bm y)$. We also denote by $\mathcal{D}$ the set of all pairs of $((\Sigma, \bm \alpha, \bm \beta), D)$, where the first element in the pair is some Heegaard diagram and the second is some domain in this Heegaard diagram.

\begin{defn}\label{shadow}
 Given a Whitney disk $\varphi \in \pi_2(\bm x, \bm y)$ we associate a domain $D(\varphi) \in \mathcal{D}(\bm x, \bm y)$ called \emph{shadow} of $\varphi$ as follows: to a region $R$ we assign the number $n_R(\varphi)$ which is equal to the intersection number $\varphi \cdot Z_r$, where $Z_r=\{r\} \times Sym^{g-1}(\Sigma)$ for some $r$ in the interior of $R$  (note that $\varphi \cdot Z_r$ does not depend on a choice of $r$). Then set $$D(\varphi)=\sum_Rn_R(\varphi)R\in\mathcal{D}(\bm x, \bm y).$$
\end{defn} 

\smallskip{} 
Let us recall the definition of the Maslov index of a Whitney disk $\varphi \in \pi_2(\bm x, \bm y)$ where $\bm x$ and $\bm y$ are some generators of the chain complex associated with a given Heegaard diagram $(\Sigma, \bm \alpha, \bm \beta)$.
The generators $\bm x$ and $\bm y$ are regarded as points belonging to $T_{\bm \alpha} \cap T_{\bm \beta} \subset Sym^g(\Sigma)$. Let $Gr_g$ be a space of totally real $g$-dimensional subspaces of $\C^g$. Since $\varphi$ is regarded as a map from $D^2 \setminus \{i, -i\}= [0,1] \times \R$ to $Sym^g(\Sigma)$ we take the pullback of $T Sym^g(\Sigma)$ to $D^2$ which is a trivial bundle $\C^g$. At point $\bm x$ we pick a ``short'' path $\gamma_{\bm x}$ between $T_{\bm x} (T_{\bm \alpha})$ and $T_{\bm x} (T_{\bm \beta})$ in $Gr_g$  and also construct in the same fashion a path $\gamma_{\bm y}$ for $y$ (by ``short'' we mean a path in $\mathcal{P}^-(Gr_g)$ in Seidel's terminology, see \cite[Section 11]{Seidel2008}).
The map from ${0}\times \R$ to $T_{\bm \alpha}$ assigns a path $\gamma_0$ in $Gr_g$ by considering pullback of $T(T_{\bm \alpha}) \subset TSym^g(\Sigma)$. Analogously, we get a path $\gamma_1$. Finally, the Maslov index $\mu(\varphi)$ is equal to a composition of all these four paths 
\[ \mu(\varphi)=[\gamma_{\bm x} \circ \gamma_{0} \circ \gamma_{\bm y} \circ \gamma_1] \in H_1(Gr_g) = \Z.\]


We also introduce \emph{combinatorial index} of any domain $D \in \mathcal{D}(\bm x, \bm y)$. 
For $x_i \in \bm x$ define  $\widetilde{\mu}_{x_i}(D)$ as the average of the coefficients of $D$ in the $4$ regions to which $x_i$ belongs. Then the \emph{point measure} of $D$ at $\bm x$ is $\widetilde{\mu}_{x}(D)=\sum_{i=1}^g\widetilde{\mu}_{x_i}(D)$.
For a $2$-chain $D$ the Euler measure $D$ is $\frac{1}{2\pi}$ of the integral of the curvature of the metric of $\Sigma$ over $D$. It is equal to the $2$-cochain that assigns $\frac{1}{2}(2-n)$ to a $2n$-gon region.
For a domain $D \in \mathcal{D}(\bm x, \bm y)$ we assign its combinatorial index by
\[
\widetilde{\mu}(D)= \widetilde{\mu}_{\bm x}(D)+\widetilde{\mu}_{\bm y}(D)+e(D).
\]

\subsection{Transformations of Heegaard diagrams}
We introduce two types of transformations on a given Heegaard diagram $(\Sigma, \bm \alpha, \bm \beta)$ that are used in this paper.
\begin{defn}
Let $\gamma$ be an oriented path in $\Sigma$ which is transverse to $\bm \alpha$ and $\bm \beta$  and is disjoint from $\bm \alpha \cap \bm \beta$; see Figure~\ref{fm}. Also assume the endpoints $p_0$ and $p_1$ of $\gamma$ do not belong to $\bm \alpha \cup \bm \beta$. Let $U$ be a small neighborhood of $\gamma$. Let $\psi$ be an isotopy of $\Sigma$ supported on a small neighborhood of $\gamma$ which moves $\bm \alpha$ curves in the direction of $\gamma$ as given in Figure~\ref{fm}. The \emph{finger move} on the $\bm \alpha$ curves along the curve $\gamma$ is a restriction $\psi|_\alpha$; this does not move the $\bm \beta$ curves. Analogously we define the \emph{finger move on the $\bm \beta$ curves}.
\end{defn}

A transition from a Heegaard diagram $(\Sigma, \bm \alpha, \bm \beta)$ to the diagram $(\Sigma, \psi(\bm \alpha), \bm \beta)$ will be called by \emph{performing a finger move on $\bm \alpha$ along $\gamma$}. Finger moves were introduced in \cite{SarWang2010}.

\begin{defn}
Given a domain $D$, its \emph{image} $D'=\psi(D)$ under a finger move $\psi$ is defined as follows: first decompose a finger move $\psi$ into a sequence $\psi= \psi_1 \circ \ldots \circ \psi_m$ of finger moves where under any $\psi_i$ only two new points of intersection between $\bm \alpha$ and $\bm \beta$ appear (see Figure~\ref{fmdomain}). Let $R_1$, $R_2$ and $R_3$ be regions at which $\psi_1 \circ \ldots \circ \psi_{i-1} (D)$ has coefficients $a, b$ and $c$, respectively, and let the finger move $\psi_i$ be as shown in the picture. Then for new regions $R_1', R_2', R_3', R_4'$ and $R_5'$ shown in Figure~\ref{fmdomain} the coefficients are $a, b, c, b$ and $a+c-b$ as shown. Repeating this procedure gives us $\psi(D)$.
\end{defn}

\begin{figure}
    \centering
    \subfloat{{\includegraphics[width=4.5cm]{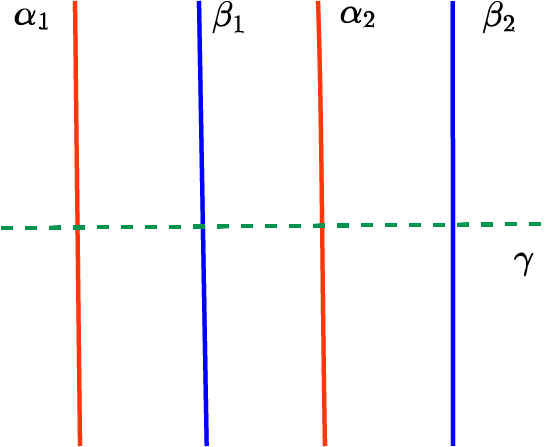} }}
    \qquad
    \subfloat{{\includegraphics[width=4.5cm]{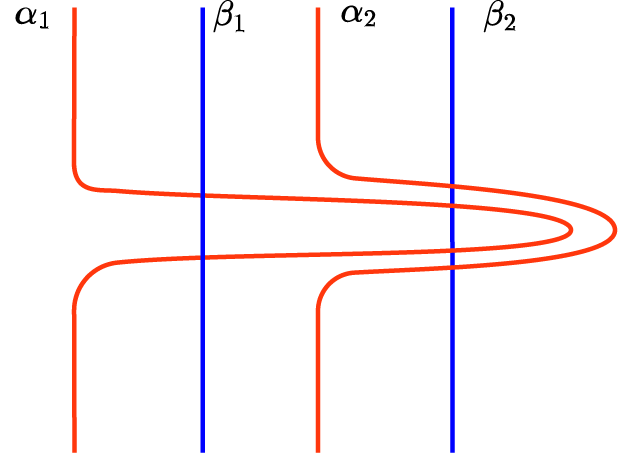} }}
    \caption{\bf Finger move}
    \label{fm}
    \subfloat{{\includegraphics[width=3.5cm]{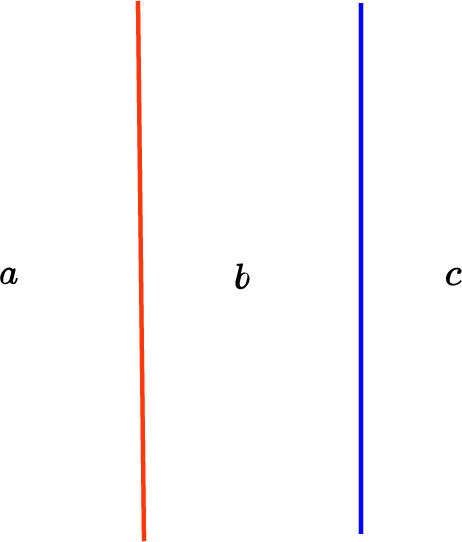} }}
    \qquad
    \subfloat{{\includegraphics[width=6cm]{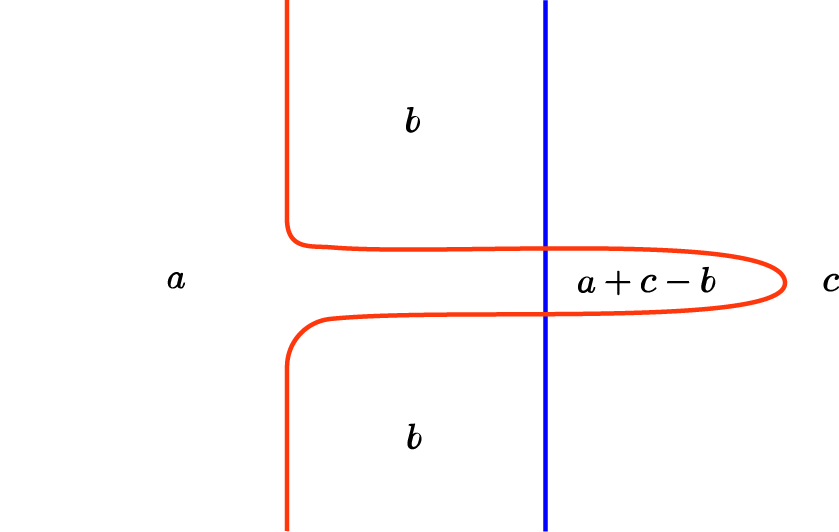} }}
    \caption{\bf Image of a domain.}
    \label{fmdomain}
\end{figure}

\begin{defn}
Let $D$ be a domain of a Heegaard diagram $(\Sigma, \bm \alpha, \bm \beta)$. We call an \emph{empty stabilization} of this Heegard diagram with respect to $D$ a stabilization which is obtained by taking the connected sum with the standard genus $1$ Heegaard diagram for $S^3$ where the attaching disk belongs to a region of $(\Sigma, \bm \alpha, \bm \beta)$ where the coefficient of $D$ is equal to $0$. The new Heegaard diagram $(\Sigma', \bm \alpha', \bm \beta')$ is also called an \emph{empty stabilization} of $(\Sigma, \bm \alpha, \bm \beta)$. The image of $D$ under an empty stabilization is $D$ itself in the new diagram.
\end{defn}

If we are given a Whitney disk $\varphi \in \pi_2(\bm x, \bm y)$ we may define its \emph{image} $\varphi'$ under a finger move by isotoping $\varphi$ with the accordance to the induced isotopy of $T_{\bm \alpha}$ inside $Sym^g(\Sigma)$.

We say that the stabilization is empty with respect to a Whitney disk $\varphi$ if it is empty with respect to $D(\varphi)$.
For the empty stabilization with respect to $\varphi$ the image of $\varphi$ is just $\varphi \times z$ where $z$ is the intersection point of added $\alpha_{g+1}$ and $\beta_{g+1}$.

We will extensively use the invariance of the Maslov index and of $\mu$ under these two types of transformations of a Heegaard diagram which is shown in the lemma below.

\begin{lemma}
Let $D \in \mathcal{D}(\bm x, \bm y)$ be a domain in Heegaard diagram $(\Sigma, \bm \alpha, \bm \beta)$ and let $\varphi \in \pi_2(\bm x, \bm y)$ be a Whitney disk. Let $D'$ be the image of $D$ under a finger move or an empty stabilization with respect to $D$. Let $\varphi'$ be the image of $\varphi$ under a finger move or an empty stabilization with respect to $\varphi$. Then $\mu(\varphi)=\mu(\varphi')$ and $\widetilde{\mu}(D)=\widetilde{\mu}(D')$. 
\end{lemma}
\begin{proof}
First, consider the case when this transformation is an empty stabilization with respect to $D$ (or $\varphi)$. Let $\alpha_{g+1}$ and $\beta_{g+1}$ be two new curves of $(\Sigma, \bm \alpha ', \bm \beta ')$ and $z$ be their point of intersection. We have $\bm x'= \bm x \cup \{z\}$, $\bm y'= \bm y \cup \{z\}$ and $D'=D \in \mathcal{D}(\bm x', \bm y')$ (or $\varphi' \in \pi_2(\bm x, \bm y))$. First, $\widetilde{\mu}(D)=\widetilde{\mu}(D')$, because $\widetilde{\mu}_z(D)=0$.
As for the Maslov index, notice that paths $\gamma_0$ and $\gamma_1$ in $Gr_g$ rise to paths $\gamma_0'$ and $\gamma_1'$ in $Gr_{g+1}$ obtained by taking direct sum with one-dimensional real subspaces of $\R^{2(g+1)}=\C^{g+1}$ corresponding to $T_z \alpha_{g+1}$ and $T_z \beta_{g+1}$ respectively. Hence $\mu(\varphi)=\mu(\varphi')$.

If the transformation is a finger move then the Maslov indices are equal because finger move keeps $\bm \alpha$ and $\bm \beta$ unchanged near $\bm x$ and $\bm y$. Hence, inspecting the definition of Maslov indices we see that only the paths in $Gr_g$ coming from the boundary $\partial D^2$ differ, and they are isotopic to the initial ones, so the class in $H_1(Gr_g)$ does not change. 
Finally, $\widetilde{\mu}(D)=\widetilde{\mu}(D')$ is immediate from the formula (see also \cite[Theorem 3.4]{Sar2011}).
\end{proof}

\subsection{Additivity of indices}
 Our strategy for proving Lipshitz's formula is based on decomposing any domain into the composition of trivial pieces: bigons and rectangles.
 To apply this decomposition we need the additivity of the Maslov index $\mu$ and of the combinatorial index $\widetilde{\mu}$ under the composition of domains.
 We know that the Maslov index $\mu$ is additive. It was shown in more generality in \cite{Sar2011} that $\widetilde{\mu}$ is additive and here we repeat the proof in our terms.

\begin{defn}
Let $\gamma_1$ and $\gamma_2$ be oriented $1$-chains in $\Sigma$ intersecting transversely. Denote by $\gamma_1 \cdot \gamma_2$ intersection number of these $1$-chains, which is equal to the signed count of intersection points where the sign is defined by comparing the orientation of $\Sigma$ and orientation coming from $\gamma_1$ and $\gamma_2$. A contribution to intersection number at an endpoint is given by a fraction $\pm \frac{1}{2}$ or $\pm \frac{1}{4}$  as in \cite{Sar2011}.
\end{defn}

\begin{lemma}\label{add}
Let $D \in \pi_2(\bm x, \bm y)$ and $D'$ be any other $2$-chain in the Heegaard diagram $(\Sigma, \bm \alpha, \bm \beta)$. Then
$\widetilde{\mu}_{\bm x}(D')-\widetilde{\mu}_{\bm y}(D')= \partial D'|_{\bm \alpha} \cdot \partial D|_{\bm \beta}=\partial D'|_{\bm \beta} \cdot \partial D|_{\bm \alpha}$.
\end{lemma}

\begin{proof}
Given orientations on $\bm \alpha$ at each intersection point $p$ between $\bm \alpha$ and $\bm \beta$ we give numbers to regions from $\mathit{I}$ to $\mathit{IV}$. Namely, $\mathit{I}$ and $\mathit{II}$ quadrants are in the upper half of $\bm \beta$ (the half in the positive direction of $\bm \alpha$) and $\mathit{I}$st quadrant is the one for which the orientation induced from $\Sigma$ is opposite to the orientation of $\bm \alpha$ at $p$. Quadrants $\mathit{III}$ and $\mathit{IV}$ are defined follow $\mathit{II}$ in the counterclockwise order on $\Sigma$.

Let us pick a point $p$ in the first quadrant near the point $x_i \in \bm x$ lying on some $\alpha_k$ and travel parallely to $\partial D|_{\bm \alpha_k}$ until we reach a first quadrant near $y_j \in \bm y$ where $\partial(\partial D|_{\bm \alpha_k})=y_j-x_i$. Then let us track the coefficient of $D'$ at each point along this path. Each time we change the region passing through $\bm \beta$ this coefficient changes by the intersection number of this small portion of our path with $\partial D'|_{\bm \beta}$. Hence, the difference between the coefficients is equal to $\partial D|_{\alpha_k} \cdot \partial D'|_{\bm \beta}$. 

Now, by taking such paths for all points in $\bm x$ and all $4$ quadrants for each of these points and then averaging we get
$\widetilde{\mu}_{\bm x}(D')-\widetilde{\mu}_{\bm y}(D')= \partial D|_{\bm \alpha} \cdot \partial D'|_{\bm \beta}$
\end{proof}

\begin{cor}
Let $D \in \pi_2(\bm x, \bm y)$ and $D' \in \pi_2(\bm y, \bm z)$. Then $\widetilde{\mu}(D*D')=\widetilde{\mu}(D)+\widetilde{\mu}(D')$.
\end{cor}

\begin{proof}

First, we easily see that
\[
\widetilde{\mu}(D*D')- \widetilde{\mu}(D)-\widetilde{\mu}(D')=\widetilde{\mu}_{\bm x}(D')+\widetilde{\mu}_{\bm z}(D)-\widetilde{\mu}_{\bm y}(D')-\widetilde{\mu}_{\bm y}(D).
\]

Second, applying Lemma~\ref{add} we get $\widetilde{\mu}_{\bm x}(D')-\widetilde{\mu}_{\bm y}(D')=\partial D|_{\bm \alpha} \cdot \partial D'|_{\bm \beta}$ and $\widetilde{\mu}_{\bm y}(D)-\widetilde{\mu}_{\bm z}(D)=\partial D|_{\bm \alpha} \cdot \partial D'|_{\bm \beta}$.

Hence
\[
\widetilde{\mu}(D*D')- \widetilde{\mu}(D)-\widetilde{\mu}(D')=0.
\]
\end{proof}

\section{Main theorem}\label{main}

\begin{thm}\label{decomp}
For a given domain $D$ in a Heegaard diagram $(\Sigma, \bm \alpha, \bm \beta)$ there is a sequence of finger moves and empty stabilizations such that in the new Heegaard diagram the image of $D$ can be represented as a composition of bigons, rectangles and their negatives.
\end{thm}

Before proving this theorem we show how Theorems~\ref{unique} and~\ref{Lipshitz} follow from Theorem~\ref{decomp}.

\begin{proof}[Proof of Theorem~\ref{unique}]
Let $D\in \mathcal{D}(\bm x, \bm y)$. Apply Theorem~\ref{decomp} to obtain the image $D'$ of $D$ and the decomposition $D'=D_1*\ldots*D_k$ of $D'$ in the new Heegaard diagram. Here each $D_i$ is either a bigon, a rectangle, or the negative of a bigon or a rectangle. Since $\overline{\mu}$ is additive we may compute $\overline{\mu}(D')=\overline{\mu}(D_1)+\ldots+\overline{\mu}(D_k)$. For each $D_i$ the value $\overline{\mu}(D_i)$ can be easily inferred from the $4$ axioms.
\end{proof}

\begin{proof}[Proof of Theorem~\ref{Lipshitz}]
We are given a Whitney disk $\varphi \in \pi_2(\bm x, \bm y)$ and we may assume that $g(\Sigma)>1$ by applying an empty stabilization if necessary. By Theorem~\ref{decomp} there is a transformation of the given Heegaard diagram such that there is a decomposition $D(\varphi')=D_1*\ldots*D_k$, where $\varphi'$ is an image of $\varphi$ under this transformation; equivalently $0=D(\varphi')*(-D_1)*\ldots*(-D_k)$. Since each $D_i$ is either a bigon or a rectangle (possibly negative) there exists a corresponding Whitney disk $\varphi_i$ such that $D(\varphi_i)=-D_i$. Then $\varphi'*\varphi_1*\ldots*\varphi_k \in \pi_2(\bm x,\bm x)$
and, moreover, $D(\varphi'*\varphi_1*\ldots*\varphi_k)=0$. From \cite[Proposition 2.15]{OzSz2004b} it now follows that $\varphi'*\varphi_1*\ldots*\varphi_k=0\in \pi_2(\bm x, \bm x)$, so $\mu(\varphi'*\varphi_1*\ldots*\varphi_k)=0$. Hence, $\mu(\varphi')=-(\mu(\varphi_1)+\ldots+\mu(\varphi_k))$.

It now suffices to show that $\mu(B)=\widetilde{\mu}(B)=1$ and $\mu(R)=\widetilde{\mu}(R)=1$ for any bigon $B$ and any rectangle $R$. For a bigon $\mu(B)=1$ as it agrees with the dimension of the (regular) moduli space of biholomorphisms from $B$ to the strip $[0, 1] \times \R$. The fact that $\mu(R)=1$ is shown in \cite[Section 8.4]{OzSz2004b}. Finally, $\widetilde{\mu}(R)=\widetilde{\mu}(B)=1$ by direct computation.
\end{proof}

Now we prove the main theorem. 

\begin{proof}[Proof of Theorem~\ref{decomp}]
In the proof we abuse notation and denote the image of a domain $D$ under any empty stabilization or a finger move by same letter $D$ .

We break our proof into several steps starting with a general domain in Step 1 and simplifying it gradually by transformations of Heegaard diagrams. Before we start dealing with general domains we need additional preparation which we perform at Step 0.

{\bf Step 0. Quadrilaterals.}

\begin{figure}
    \centering
    \subfloat{{\includegraphics[width=3.5cm]{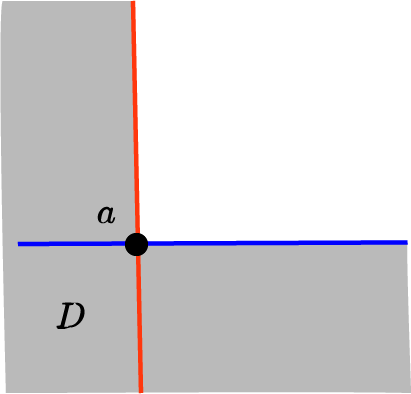} }}
    \qquad
    \subfloat{{\includegraphics[width=3.5cm]{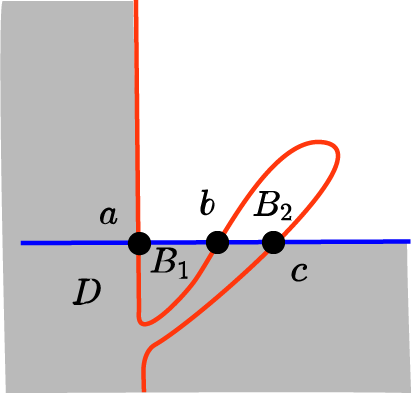} }}
    \caption{\bf $270^{\circ}$ angle}
    \label{270}
\end{figure}
Before considering the most general domains we need to take care of another type of ``building block". Namely, let $D$ be an embedded domain with non-negative coefficients which is homeomorphic to a disk and has a boundary consisting of four sides. We would call such $D$ a \emph{quadrilateral}. If all angles of $D$ are $90^\circ$ then it is already a rectangle. Otherwise, an angle at one of its vertices $a$ is $270^\circ$.  Consider a finger move shown in Figure~\ref{270} which creates two bigons $B_1 \in \pi_2(a,b)$ and $B_2 \in \pi_2(b,c)$. Then $D*(-B_1)$ is a quadrilateral with one more $90^\circ$ angle than $D$. Proceeding in the same fashion we decompose $D$ into a rectangle and several bigons.

Henceforth, in the following steps it is enough to decompose any domain into quadrilaterals and bigons with arbitrary angles.

\smallskip
{\bf Step 1. Making the boundary embedded.}

First, we show that after several finger moves we may compose $D$ with some bigons (or their negatives) and obtain a domain $D'$ with embedded boundary.

\begin{figure}
    \centering
    \subfloat{{\includegraphics[width=7cm]{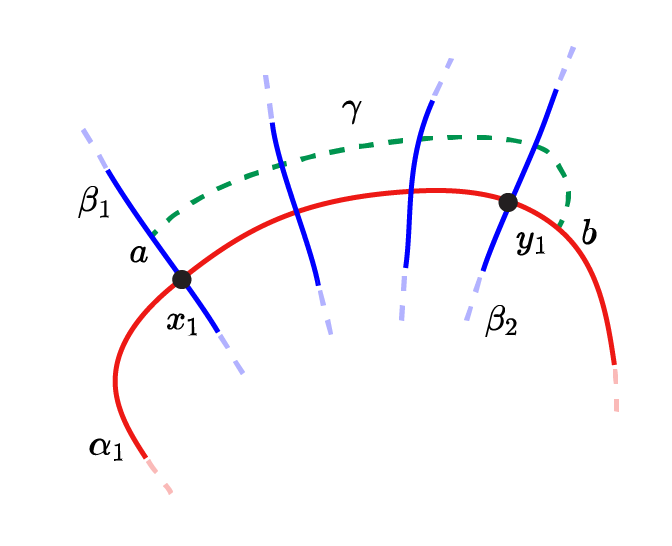} }}
    \qquad
    \subfloat{{\includegraphics[width=7cm]{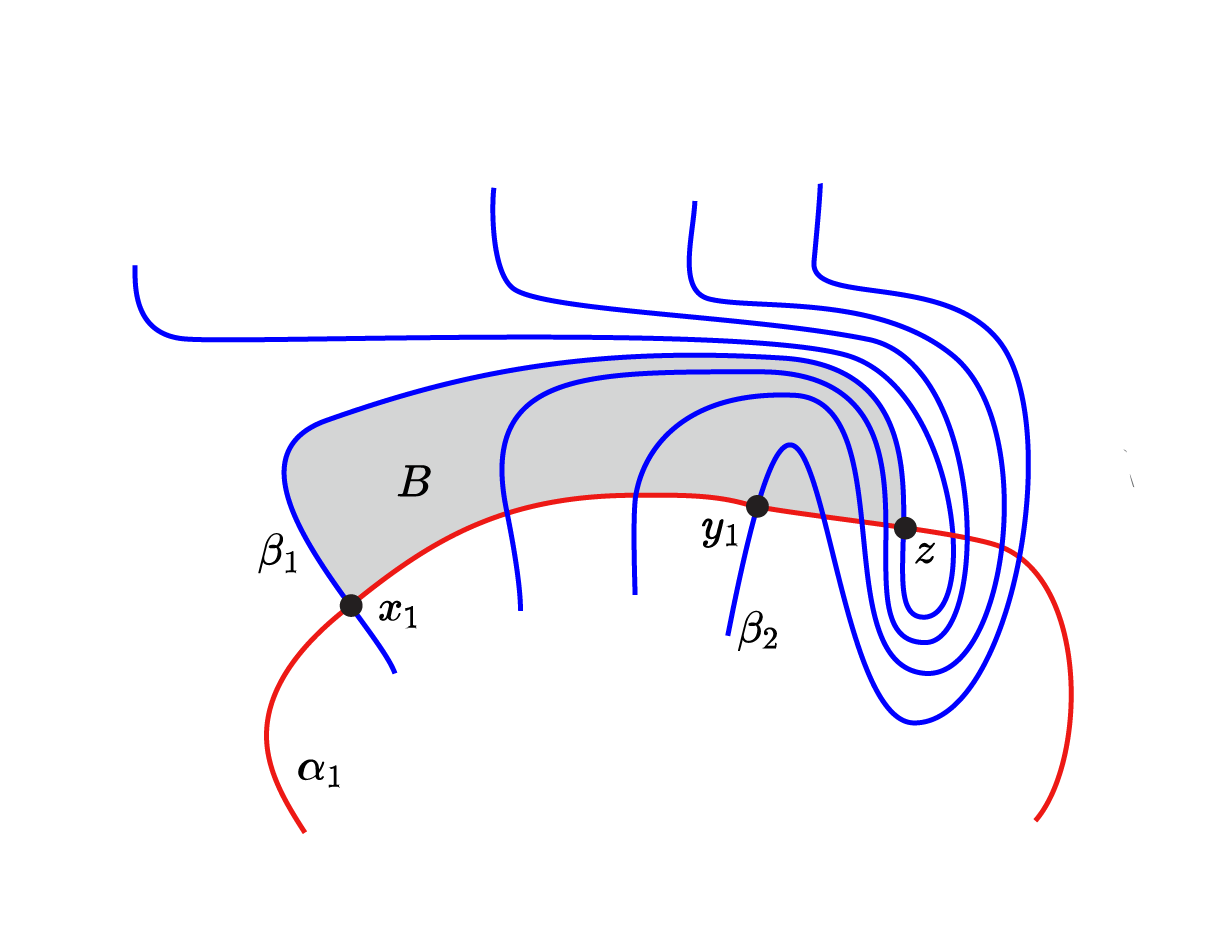} }}
    \caption{\bf Making the boundary embedded}
    \label{embed}
\end{figure}

Assume without loss of generality that $\partial(\partial D|_{\alpha_1})=y_1-x_1$ with respect to the chosen orientation on $\alpha_1$. Here $x_1$ lies on $\beta_1$ and $y_1$ lies on $\beta_2$ (where $\beta_2$ may be equal to $\beta_1$). We may assume that the positively oriented embedded arc from $x_1$ to $y_1$ in $\alpha_1$ is covered $k+1$ times by $\partial D$ and the positively oriented embedded arc from $y_1$ to $x_1$ is covered $k$ times. For now assume $k>0$. 

Since the $\bm \alpha$ and $\bm \beta$ curves are oriented, for any region we may distinguish whether it lies to the left of a given $\bm \alpha$ or $\bm \beta$ curve. Let $a$ be a point very close to $x_1$ on $\beta_1$ to the left of $\alpha_1$ (see Figure~\ref{embed}). Also let $b$ be a point on $\alpha_1$ near $y_1$ not belonging to the positive arc from $x_1$ to $y_1$. Let $\gamma$ be (a slight extension of) a curve starting at $a$ and parallel to $\alpha_1$ until it hits $\beta_2$ and then we connect it with $b$. We make a finger move on $\bm \beta$ curves along the curve $\gamma$ creating new points of intersection of $\bm \beta$ curves intersecting $\gamma$ (including $\beta_1$ and $\beta_2$ ) with $\alpha_1$. Let $z$ be the new intersection point of $\beta_1$ and $\alpha_1$ which is the closest to $x_1$ on $\beta_1$.

As a result, we created a bigon $B$ connecting $x_1$ to $z$. Then  $D'=D*(-B) \in \mathcal{D}(\{z, x_2, \dots, x_g\}, \{y_1, y_2, \ldots, y_g\})$. Note that in $\partial D'$ the arc from $z$ to $y_1$ is covered $k$ times and the arc from $y_1$ to $z$ is covered $(k-1)$ times.

In the case $k<0$ we would draw $\gamma$ parallel to the negative embedded arc from $x_1$ to $y_1$ and proceed analogously. Repeating this procedure we first make sure that $\partial D \cap \bm \alpha$ is embedded and then we repeat it for $\bm \beta$ and end up with a domain having an embedded boundary.


\smallskip
{\bf Step 2. Making the boundary connected.}

\begin{figure}
    \centering
    \subfloat{{\includegraphics[width=7cm]{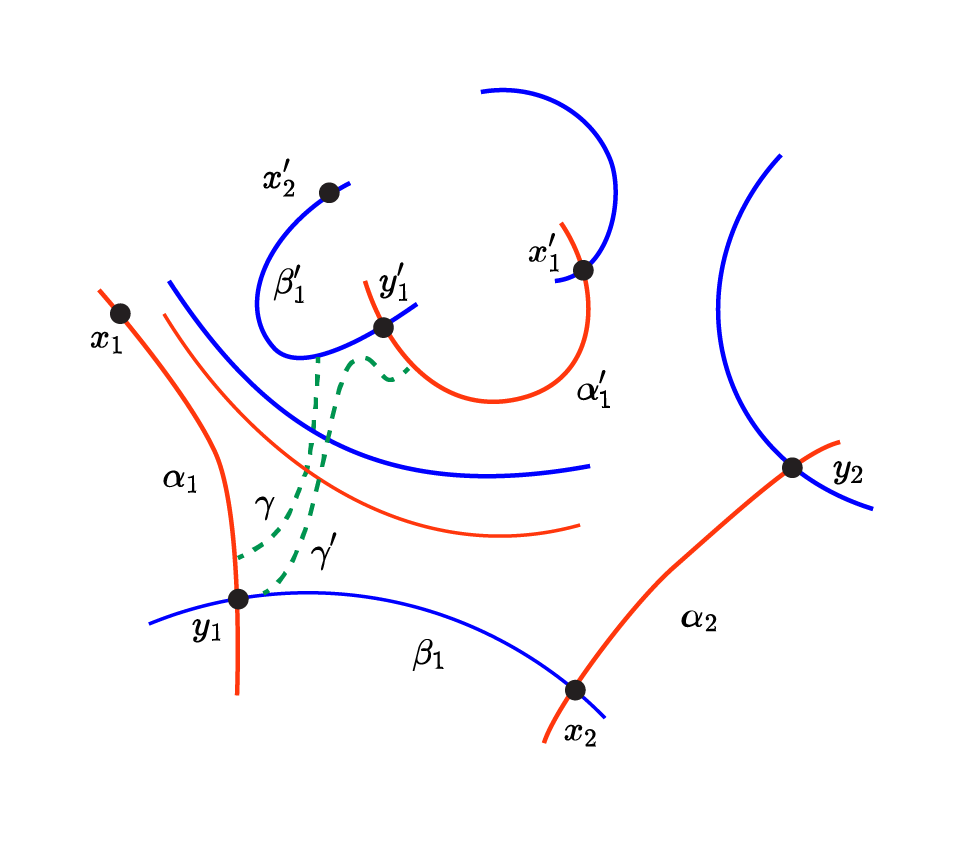} }}
    \qquad
    \subfloat{{\includegraphics[width=7cm]{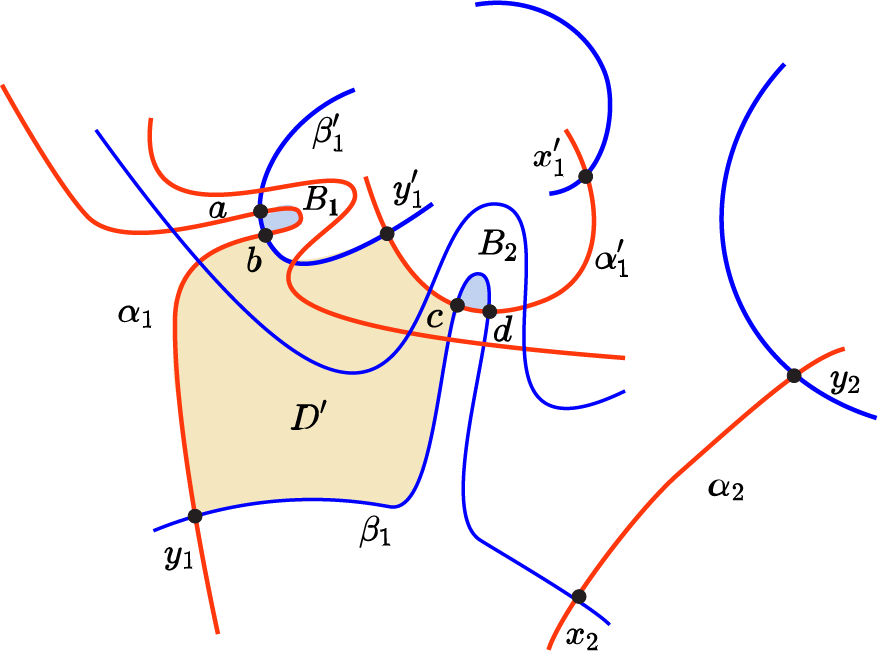} }}
    \caption{\bf Making boundary connected}
    \label{connected}
\end{figure}
Assume the boundary of a domain $D$ consists of more than one closed curve as shown in Figure~\ref{connected}. 
Suppose $\partial D \cap (\alpha_1 \cap \beta_1)=y_1$, $\partial D \cap (\beta_1 \cap \alpha_2)=x_2$ and $\partial(\partial D|_{\alpha_2})=y_2-x_2$. Let us pick a point on $\alpha_1$ close to $y_1$ and let $\gamma$ be an arc connecting it to a point on the restriction of some other component of $\partial D$ to $\bm \beta$ such that $int(\gamma)$ does not intersect $\partial D$. We denote  by $\beta_1'$ the curve to which the endpoint of $\gamma$ belongs and $\partial(\partial D|_{\beta_1'})=x_2'-y_1'$ while $\partial(\partial D|_{\alpha_1'})=y_1'-x_1'$. We also can make sure that $\gamma$ intersects $\beta_1'$ at a point near $y_1'$. We may decompose $\gamma=\gamma_1\gamma_2\gamma_3$ where $\gamma_1$ and $\gamma_2$ are small portions near the ends and $\gamma_2$ is the remaining portion in the middle.

We draw a curve $\gamma_1'$ connecting a point near $y_1$ on $\beta_1$ with some point near the end of $\gamma_1$ such that $\gamma_1'$ does not intersect $\gamma_1$. Starting at the end of $\gamma_1'$, we draw a curve $\gamma_2'$ parallel to $\gamma_2$. Then we connect the end of $\gamma_2$ by a curve $\gamma_3$ parallel to $\beta_1'$ with a point on $\alpha_1'$ near $y_1'$. We set $\gamma'=\gamma_1'\gamma_2'\gamma_3'$.

Now we make a finger move on $\bm \alpha$ curves along $\gamma$ and we also make a finger move on $\bm \beta$ curves along $\gamma'$. Denote the new points of intersection between $\alpha_1$ and $\beta_1'$ by $a$ and $b$, and analogously denote the points of intersection between $\beta_1$ and $\alpha_1'$ by $c$ and $d$. 

Denote the new quadrilateral by $D'\in \mathcal{D}(\{b, c\},\{y_1, y_1'\})$ and the two new bigons by $B_1\in \mathcal{D}(\{a\},\{b\})$ and $B_2\in \mathcal{D}(\{c\},\{d\})$. 

Then consider 
\[ D*(-D')*B_1*B_2 \in \mathcal{D}(\bm x,  \{a,y_2',...,d,y_2,...\})
\]
whose boundary contains one fewer component than $D$. From here we may proceed inductively on the number of boundary components.
\begin{figure}
    \centering
    \subfloat{{\includegraphics[width=6cm]{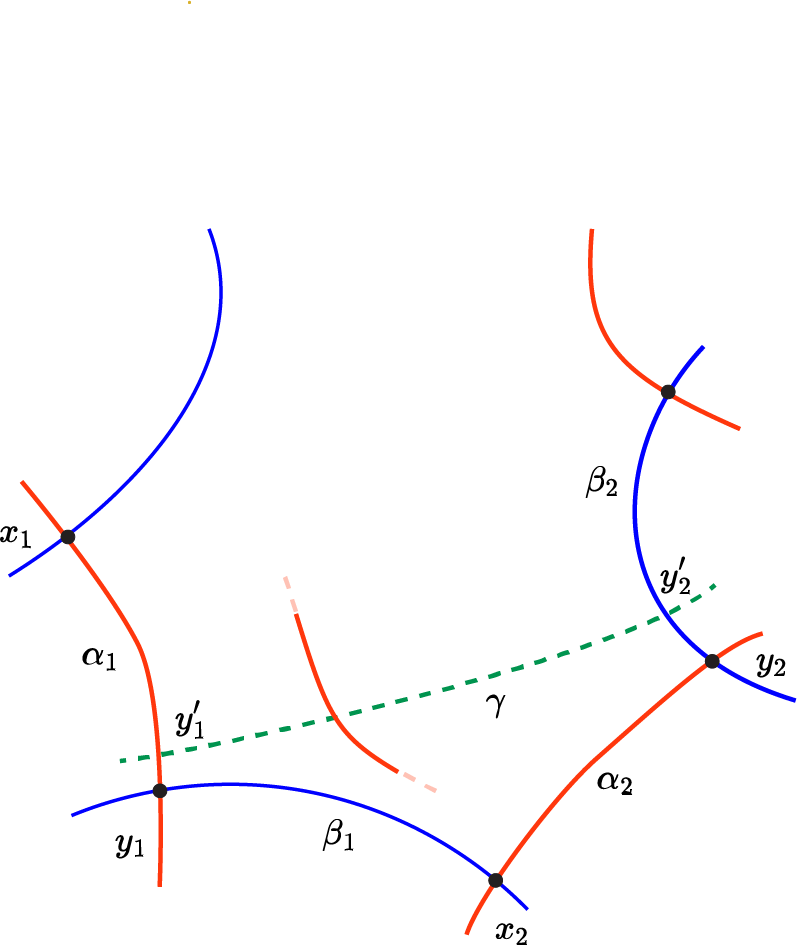} }}
    \qquad
    \subfloat{{\includegraphics[width=7cm]{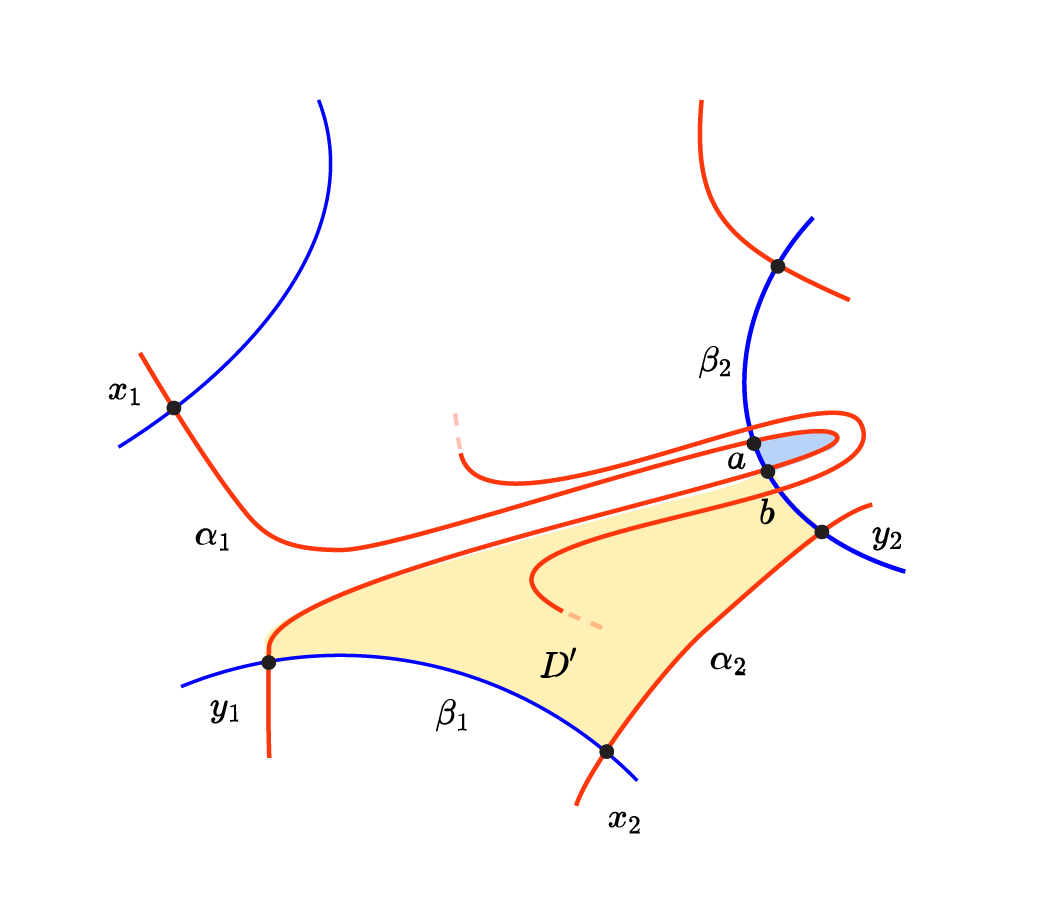} }}
    \caption{\bf Decreasing the number of sides}
    \label{2ngon}
\end{figure}

\smallskip
{\bf Step 3. Reducing to a quadrangle or a bigon boundary.}
Let $D$ be a domain and applying Steps 1 and 2 we assume that $\partial D$ is embedded and connected. 

Let the boundary of a domain $D$ be $2n$-sided, i.e. is of the form $x_1y_1x_2y_2x_3\ldots y_n$ with $n>2$ where $x_iy_i$ is an arc on $\bm \alpha$ and $y_ix_{i+1}$ is an arc on $\bm \beta$. Assume that the first $4$ sides lie on $\alpha_1$, $\beta_1$, $\alpha_2$ and $\beta_2$ (see Figure~\ref{2ngon}). Let us pick a point $y_1'$ on $x_1y_1$ near $y_1$ and a point $y_2'$ on $\beta_2$ near $y_2$. Let $\gamma$ be an arc connecting $y_1'$ and $y_2'$ such that $int(\gamma) \cap \partial D= \emptyset$ and the arcs $\gamma$, $y_1y_1'$, $y_1x_2$, $x_2y_2$ and $y_2y_2'$ bound an embedded disk. We make a finger move on $\bm \alpha$ curves along $\gamma$ and denote by $a$ and $b$ two new points of intersection between $\alpha_1$ and $\beta_2$. Notice that we created two new domains: a bigon $B \in \mathcal{D} (\{a\}, \{b\})$ and a quadrilateral $D'\in \mathcal{D} (\{x_2, b\}, \{y_1, y_2\})$. 

We may replace $D$ with a domain 
\[
D*(-D')*B \in \pi_2( \{x_1, x_2, \ldots, x_n\}, \{a, x_2, y_3, \ldots, y_n\}),\]
which has $2n-2$ sides. 
\begin{figure}
    \centering
    \subfloat{{\includegraphics[width=6cm]{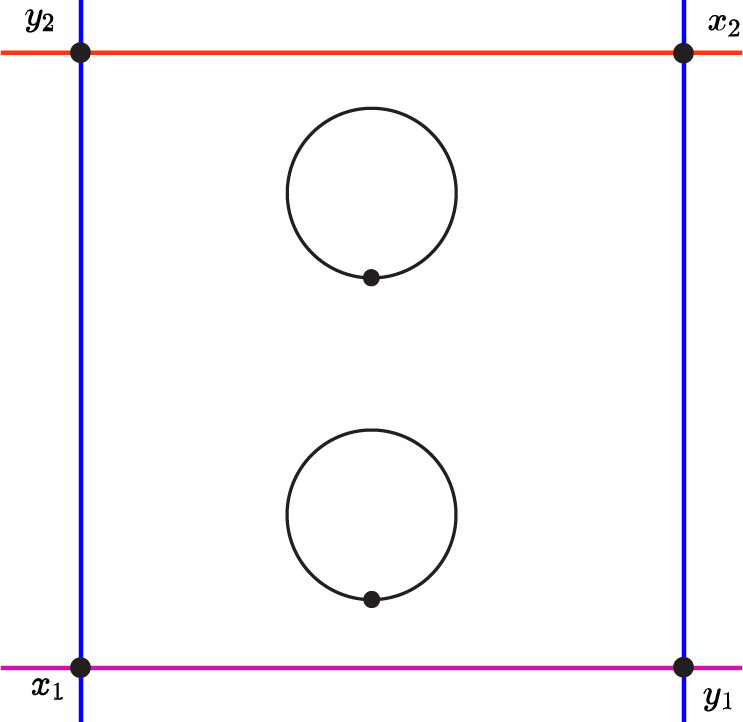} }}
    \qquad
    \subfloat{{\includegraphics[width=8cm]{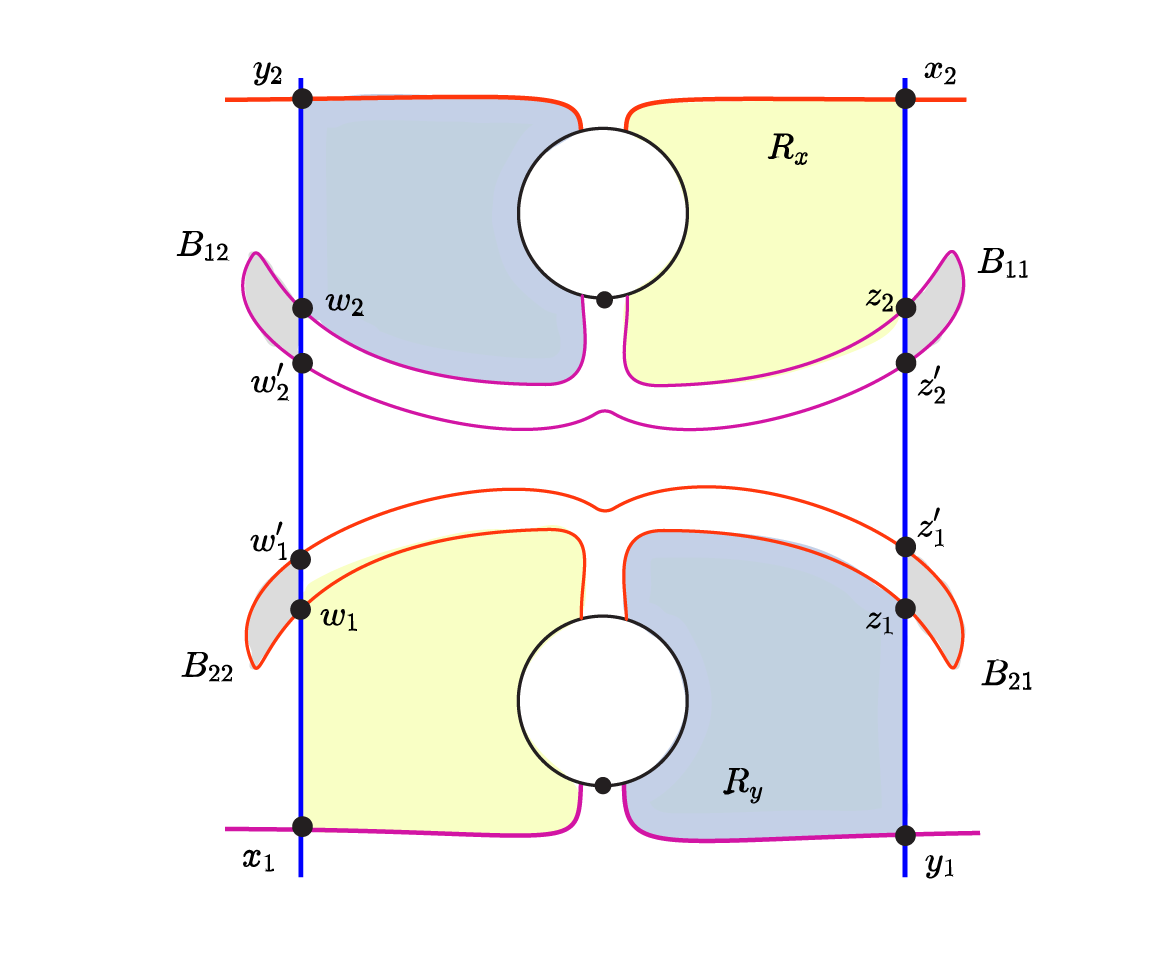} }}
    \caption{\bf 4-sided boundary.}
    \label{palms}
\end{figure}

\smallskip
{\bf Step 4a. 4-sided boundary.}

Given two domains which have the same 4-sided curve as the boundary, their difference is an element in $H_2(\Sigma)$ which is generated by $[\Sigma]$. We will prove in Claim~\ref{whole} that $\Sigma$ can be decomposed into bigons and rectangles after some finger moves and empty stabilizations. Therefore we may add $k[\Sigma]$ to a given domain $D$ to ensure that we get a domain represented by an embedded $2$-chain. This domain, which we also call $D$, is isotopic to a disk with $m$ handles.

We will now show how to decompose this handlebody $D$ into bigons and quadrilaterals.

\begin{figure}
    \centering
    \subfloat{{\includegraphics[width=8.5cm]{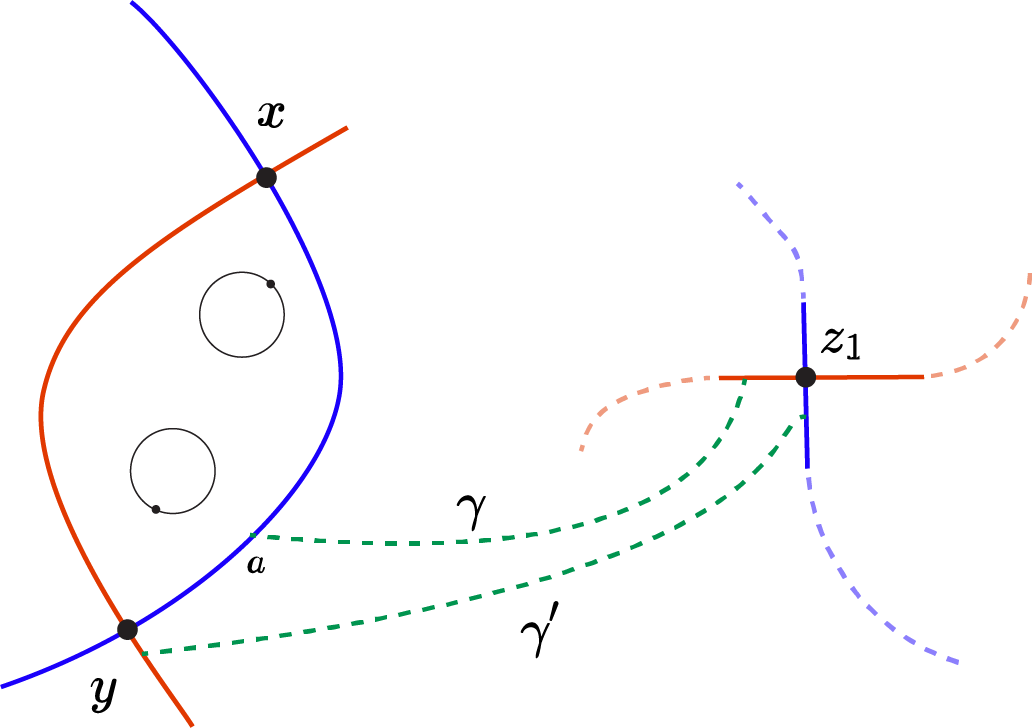} }}
    \qquad
    \subfloat{{\includegraphics[width=9cm]{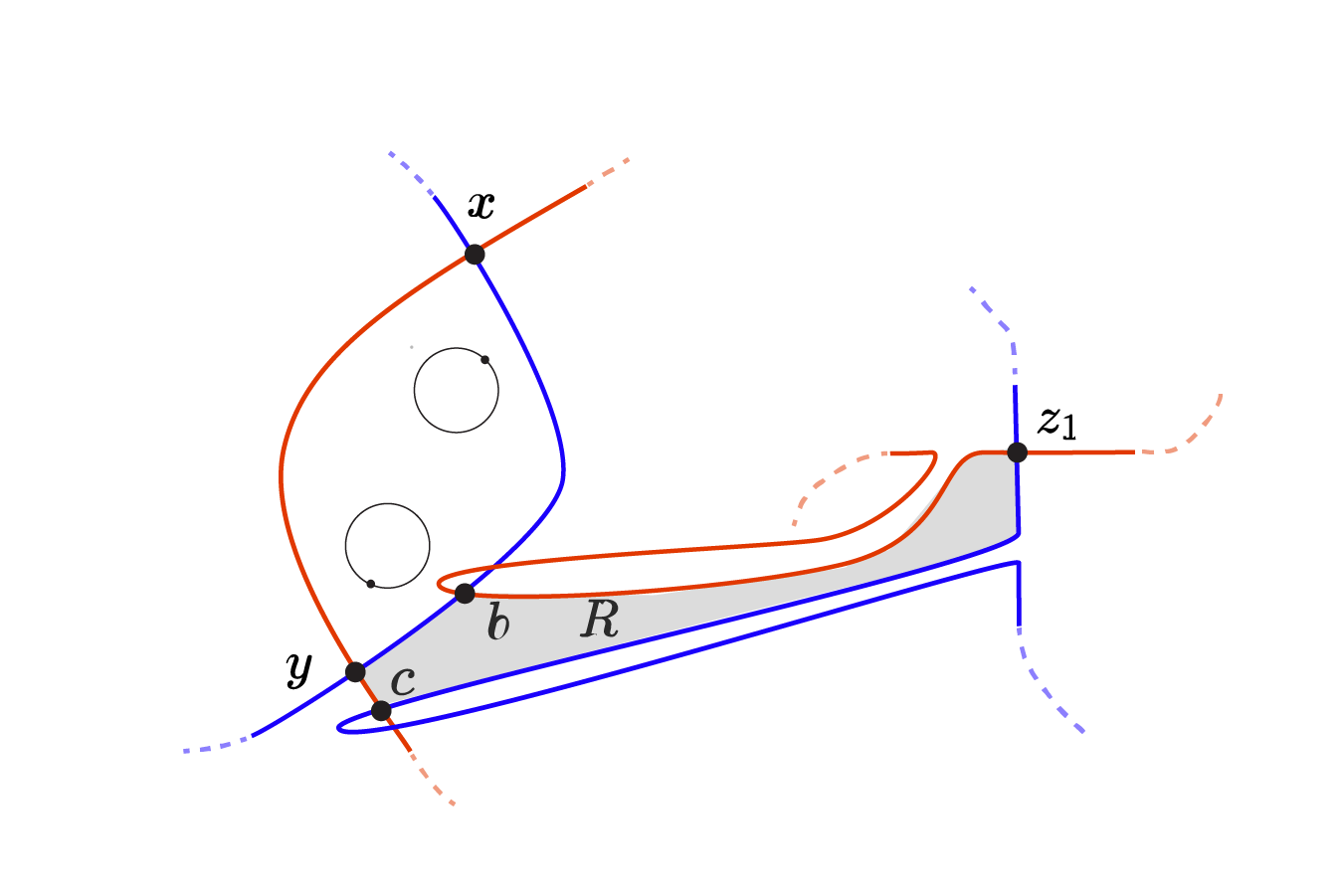} }}
    \caption{\bf 2-sided boundary}
    \label{2side}
\end{figure}
We depict handles as pairs of circles with opposite orientations and one of the handles is drawn in Figure~\ref{palms}. We now draw two ``palm trees'' as in the picture by making two finger moves on $\bm \beta$ curves and we assume that the other $m-1$ handles are in the white region ``between'' two palm trees. Let the $8$ new points of intersection between the $\bm \alpha$ and $\bm \beta$ curves be labeled as in Figure~\ref{palms}.

We now have $4$ new bigons $B_{11}, B_{12}, B_{21}$ and $B_{22}$ and $2$ new quadrilaterals $R_x$ and $R_y$. Then the domain $D*(-R_x)*(-R_y)*B_{11}* B_{12}* B_{21}*B_{22} \in \mathcal{D}(\{z_1',z_2'\},\{w_1',w_2'\})$ has $m-1$ handles and from here proceed inductively on the number of handles.

{\bf Step 4b. 2-sided boundary.}
Here we reduce this case to the $4$-sided boundary case treated in Step 4a.

As in the $4$-sided boundary case, we may assume $D$ to be an embedded 2-chain represented by a handlebody as shown in Figure~\ref{2side}. By the method of Step 0 we may assume that both angles are $90^\circ$.

We may assume that $g \ge 2$ by applying an empty stabilization with respect to $D$ if necessary. Let $D\in \mathcal{D}(\{x,z_1,\ldots, z_{g-1}\}, \{y, z_1,\ldots, z_{g-1}\})$. We may suppose $z_1\in \alpha_2 \cap \beta_2$ lies outside of $D$ by applying an empty stabilization with respect to a given $D$.
Let $\gamma$ be an arc connecting the point $z_1$ with some point $a$ on the boundary of $D$ such that $\gamma$ intersects $D$ once at the point $a$. We may assume $a$ to be somewhere on $\beta_1$. We also alter the path $\gamma$ slightly so that it starts at some point on $\alpha_2$ and $\beta_2$ is to its left near this endpoint.

Now draw $\gamma'$ almost parallel to $\gamma$ so that one of its endpoints is on $\beta_2$ near $z_1$. When $\gamma'$ reaches neighborhood of $a$ we extend it parallel to $-\beta_2$ until it hits $\alpha_1$ at some point outside of $D$ near $y$.

We make a finger move on $\bm \alpha$ along $\gamma$ and a finger move on $\bm \beta$ along $\gamma'$ creating points $b \in \alpha_2 \cap \beta_1$ and $c \in \beta_2 \cap \alpha_1$. There is now a new quadrilateral $R \in \mathcal{D}(\{y, z_1, z_2, \ldots, z_{g-1}\}, \{b, c,\ldots, z_{g-1}\})$.

Therefore, we may replace $D$ with $D*R \in \mathcal{D}(\{x, z_1, z_2, \ldots, z_{g-1}\}, \{b, c,\ldots, z_{g-1}\})$ reducing to Step 4a.



{\bf Step 5. Inner points.}
\begin{figure}
    \centering
    \subfloat{{\includegraphics[width=5cm]{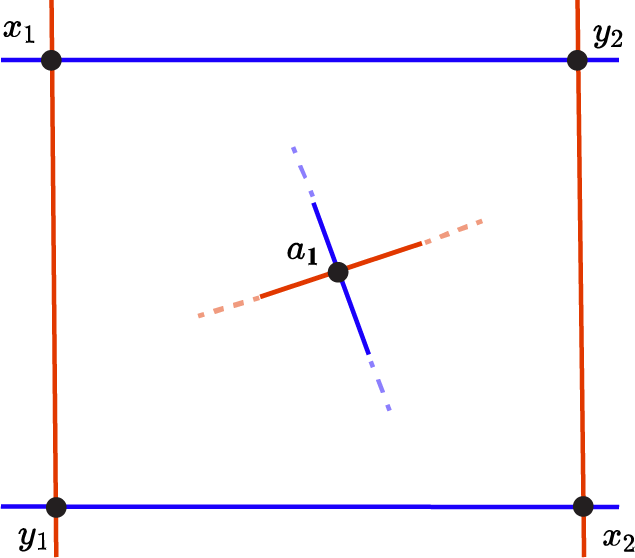} }}
    \qquad
    \subfloat{{\includegraphics[width=5cm]{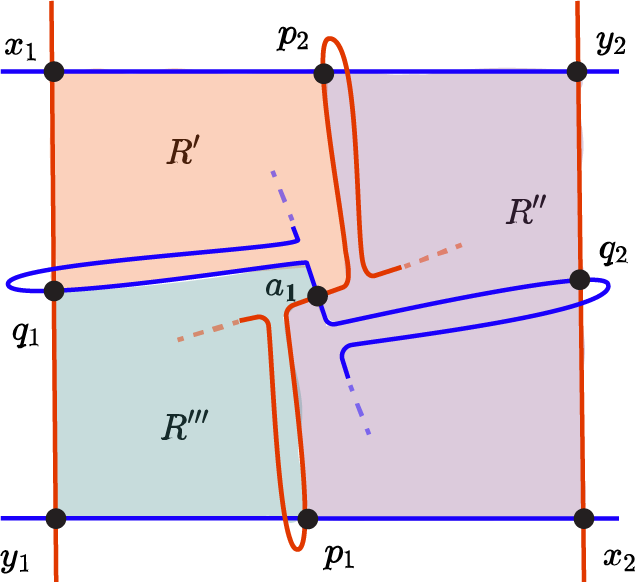} }}
    \caption{\bf Inner points}
    \label{inner}
\end{figure}

By now we have represented the initial $D$ as a composition of domains which geometrically look like bigons and rectangles, but some of the rectangles may contain some points from generators inside of them. Namely, let $R \in \pi_2({\bm x}, {\bm y})$ where $\bm x= \{x_1, x_2, a_1, \ldots a_k, b_{k+1}, \ldots, b_{g-2}\}$ and $\bm y= \{y_1, y_2, a_1, \ldots a_k, b_{k+1}, \ldots, b_{g-2}\}$ where the $a_i$'s lie inside of $R$ and the $b_j$'s lie outside of $R$.

Let $a_1 \in \alpha_3 \cap \beta_3$. We make finger moves on $\alpha_3$ and $\beta_3$ creating $2$ points of intersection between $\alpha_3$ and $\beta_1, \beta_2$ and $2$ points of intersection between $\beta_3$ and $\alpha_1, \alpha_2$ (actually, we created twice more points, but we stress our attention on these $4$). We call these points $p_1, p_2, q_1, q_2$, respectively. This is illustrated in Figure~\ref{inner}.

Then we may see that $R=R'*R''*R'''$ where 
\[
R' \in \mathcal{D}(\bm x, \{q_1,x_2, p_2, a_2 \ldots a_k, b_{k+1}, \ldots, b_g\}),
\]
\[
R'' \in \mathcal{D}(\{q_1,x_2, p_2, a_2 \ldots a_k, b_{k+1}, \ldots, b_g\}, \{q_1,y_2, p_1, \ldots a_k, b_{k+1}, \ldots, b_g\}),
\]
\[
R''' \in \mathcal{D}(\{q_1,x_2, p_1, a_2 \ldots a_k, b_{k+1}, \ldots, b_g\}, \bm y).
\]

Each of these $3$ rectangles has fewer points from generators inside than $R$ and we may proceed by induction on the number of points.

This completes the proof of Theorem~\ref{main} assuming Claim~\ref{whole} below.
\end{proof}

\begin{cl}\label{whole}
The surface $\Sigma\in \mathcal{D}(\bm x, \bm x)$ for ${\bm x}=\{x_1,\ldots, x_g\}$ is a domain which can be decomposed into bigons and rectangles after applying a sequence of finger moves. 
\end{cl}

\begin{proof}
Let us assume $g>1$. We may apply a diffeomorphism $\varphi$ to the Heegaard diagram $(\Sigma, \bm \alpha, \bm \beta)$ such that the image of $\bm x$ is a collection $\bm x'$ such that all $x_i'$ are located in some small disk region such that $x_1'$ and $x_2'$ are on the boundary of the convex hull of $\bm x'$. Let us denote by $(\Sigma, \bm \alpha', \bm \beta')$ the image of the initial diagram under $\varphi$, i.e. $\bm \alpha'= \varphi(\bm \alpha)$ and $\bm \beta'= \varphi(\bm \beta)$. We may assume that $x_1' \in \alpha_1' \cap \beta_2'$ and $x_2' \in \alpha_2' \cap \beta_1'$.
Then we make finger moves on $\alpha_1', \alpha_2'$ and on $\beta_1', \beta_2'$ as in Figure~\ref{claim} creating a rectangle $R$ with boundary $x_1'y_1'x_2'y_2'$ such that all other points of $\bm x'$ are inside of $R$. Then $\Sigma*(-R)$ connects $\bm x'$ to $\{y_1', y_2', \ldots, x_g'\}$ and it is a a region with quadrangle boundary and $g$ handles. Now we may proceed as in Step 4a of the proof of the theorem.

Let $g=1$. We can make a finger move on $\alpha_1$ which creates a bigon $B$ connecting given point $x$ to some point $x'$. Then $(-B)*\Sigma \in \mathcal{D}(\{x'\}, \{x\}) $. Now we make a stabilization inside $B$ and can proceed as in the proof of Step 4b.
 \end{proof}

\begin{figure}
    \centering
    \subfloat{{\includegraphics[width=1.8cm]{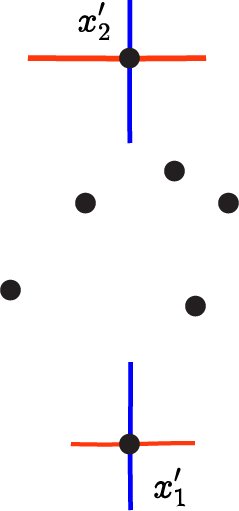} }}
    \qquad
    \subfloat{{\includegraphics[width=6cm]{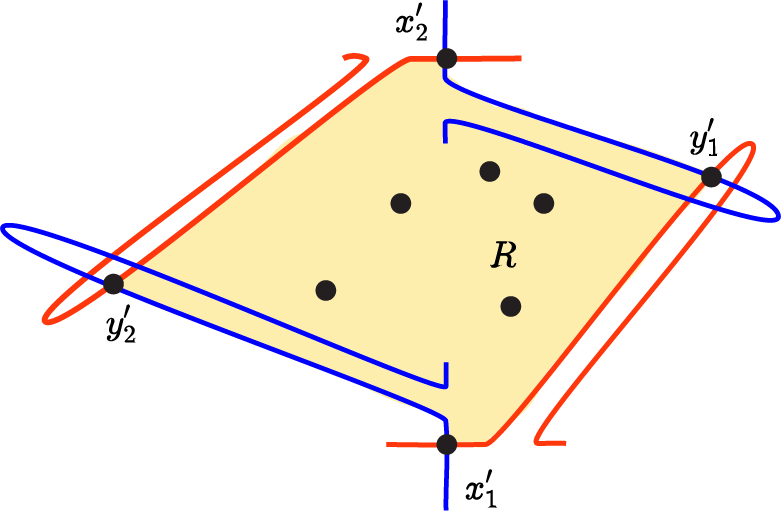} }}
    \caption{\bf Claim 3.2}
    \label{claim}
\end{figure}

\end{document}